\documentclass[a4paper, 11pt, reqno]{amsart}

\usepackage[usenames,dvipsnames]{color}
\usepackage{amssymb, amsmath, latexsym, graphics, graphicx, pstricks}

\newcommand{\ol}[1]{{\overline{#1}}}

\newtheorem{theorem}{Theorem}[section]

\newtheorem{lemma}[theorem]{Lemma}
\newtheorem{proposition}[theorem]{Proposition}

\newtheorem{algorithm}[theorem]{Algorithm}

\begin{document}
\title[The word problem for free adequate semigroups]{The word problem for \\ free adequate semigroups}

\subjclass[2010]{20M05,20M10}

\maketitle

\begin{center}

MARK KAMBITES\footnote{School of Mathematics, University of Manchester, Manchester M13 9PL, 
 England. Email \texttt{Mark.Kambites@manchester.ac.uk}.}
\ and \
ALEXANDR KAZDA\footnote{Department of Mathematics, 1326 Stevenson Center,
Vanderbilt University, Nashville, TN 37240, USA.
 Email \texttt{alex.kazda@gmail.com}.} \\

 \date{\today}

\end{center}

\begin{abstract}
We study the complexity of computation in finitely generated free left, right and two-sided
adequate semigroups and monoids. We present polynomial time (quadratic in the RAM model of
computation) algorithms to solve the word problem and compute normal forms
in each of these, and hence also to test whether any given identity holds in the classes
of left, right and/or two-sided adequate semigroups.
\end{abstract}

\section{Introduction}

Adequate semigroups form a class of semigroups in which the cancellation properties
of elements are reflected in the cancellation properties of idempotents.
They form a natural common generalisation of inverse semigroups and cancellative
monoids. Their importance was first recognised by Fountain 
in the 1970's \cite{Fountain79}, but for many years their study was restricted by
a lack of applicable methods. In the last few years, interest has been
reawakened by the development of several new techniques and results 
(see for example \cite{Araujo11,Branco09a,Gomes09,
K_freeadequate,K_onesidedadequate}).

\textit{Free algebras} form a natural focus of attention when studying any class of algebras
in which they exist; indeed, an understanding of the free objects in a class of algebras usually
yields considerable information about the class as a whole. In the case of adequate semigroups,
the fact that free adequate semigroups of every rank exist follows from
elementary principles of universal algebra (see for example \cite[Proposition~VI.4.5]{Cohn81}),
but an explicit description proved elusive until recently. In \cite{K_freeadequate}, the first
author gave a concrete geometric realisation of the free adequate semigroups (and monoids),
inspired by Munn's celebrated representation of the free inverse semigroups, in terms of directed,
labelled, birooted trees under a natural combinatorial multiplication operation. In \cite{K_onesidedadequate} he showed further
that the certain natural subsemigroups are the free objects in the
related categories of \textit{left adequate} and \textit{right adequate}
semigroups (and monoids). The free left, right and two-sided adequate semigroups also
turn out to be free objects in the larger classes of left, right and two-sided
\textit{Ehresmann semigroups} \cite{Branco09a,Gomes09,Lawson91}.

This representation immediately gave rise to a non-deterministic
polynomial-time
algorithm for the
word problem in finite rank free adequate semigroups and
monoids (and hence also in finite rank free left adequate and right adequate
semigroups and monoids). Since a relation holds in a free algebra in a category
exactly if the corresponding identity holds in \textit{all} algebras in the
category, this also yields an algorithm to check whether a given identity
holds in all adequate (or left adequate or right adequate) semigroups or
monoids.
This algorithm has proved surprisingly practical for human application to 
short words, with the intuitive geometric nature of the representation 
often allowing an effective use of guesswork to circumvent the issue
of non-determinism. However, a non-deterministic algorithm is clearly
not well-suited to computer implementation for larger words, and it would
also be more satisfactory for theoretical reasons to know the precise
asymptotic complexity of the problem.

In this paper, we apply some ideas from constraint satisfaction theory
to refine the algorithm into a deterministic form, thus showing that the
word problems for free adequate, free left adequate and free adequate
semigroups and monoids, and hence also the problem of checking whether identities hold
for all adequate semigroups or monoids, are decidable in quadratic (in the
RAM model of computation) time. Moreover, we show how to efficiently (again, in quadratic
time in the RAM model) compute normal forms (either trees
or words) for elements of the free adequate semigroup or monoid.

\section{Preliminaries}\label{sec_prelim}

In this section we very briefly recall the definitions of (left, right and two-sided) adequate
semigroups, and the first author's characterisation of the free (left, right
and two-sided) adequate semigroups monoids. The reader seeking a more complete
introduction with examples is referred to \cite{Fountain79} for adequate semigroups in general and
\cite{K_freeadequate} for free adequate semigroups and monoids.

Let $S$ be a semigroup whose idempotent elements commute. Denote by $S^1$ the
monoid consisting of $S$ with a new identity element $1$ adjoined. Then $S$ is
called \textit{left adequate} if for every element $x \in S$ there
is an idempotent element $x^+ \in S$ such that
$a x = bx \iff a x^+ = b x^+$ for all $a, b \in S^1$. If $S$ is left adequate then
the choice of $x^+$ is uniquely determined by $x$, and it is usual to
consider $S$ as a $(2,1)$-algebra with the binary operation of multiplication
and the unary operation $x \mapsto x^+$. In particular, we restrict attention
to morphisms respecting both operations.
Dually, $S$ is \textit{right adequate} if for every $x \in S$ there is an
idempotent $x^*$ with $xa = xb \iff x^* a = x^* b$ for all $a, b \in S^1$;
right adequate semigroups are also $(2,1)$-algebras. The semigroup $S$ is
called \textit{(two-sided) adequate} if it is both left adequate and right
adequate; the two maps $x \to x^+$ and $x \to x^*$, which in general will
be different, make $S$ into a $(2,1,1)$-algebra.

Now let $\Sigma$ be an alphabet. A \textit{$\Sigma$-tree} (or just
a \textit{tree} if the alphabet $\Sigma$ is clear) is a finite
directed graph with edges labelled by letters from $\Sigma$, whose
underlying undirected graph is a tree, together
with two distinguished
vertices (the \textit{start} vertex and the \textit{end} vertex) such that
there is a (possibly empty) directed path
from the start vertex to the end
vertex. The (unique) simple path from the start vertex to the end vertex
is termed the \textit{trunk} of the tree; vertices and edges lying on it are
called \textit{trunk vertices} and \textit{trunk edges} respectively. If
$e$ is an edge in such a tree, we denote by $\alpha(e)$,
$\omega(e)$ and $\lambda(e)$ respectively the source vertex, target vertex
and label of $e$. We say that a vertex $v$ is a \textit{descendant} of a vertex
$u$ if the unique simple undirected path between $v$ and the start vertex passes
through $u$.

As a notational convenience, we let $\Sigma' = \lbrace j' \mid j \in \Sigma \rbrace$
be an alphabet disjoint from and in bijective correspondence with $\Sigma$, and say
that a $\Sigma$-tree $X$ has \textit{an edge from $u$ to $v$ labelled $j'$} to mean that
it has an edge from $v$ to $u$ labelled $j$. (Intuitively, the elements of $\Sigma'$ can
be thought of as labelling directed edges when read ``in the wrong direction''. This
notation will allow a unified consideration of labels and directions of edges;
since label and direction play similar roles as obstructions to a morphism mapping one
edge to another, considering them together simplifies our arguments in several places.)

A \textit{morphism} $\rho : X \to Y$ of $\Sigma$-trees $X$ and $Y$ is a
map taking edges to edges and vertices to vertices which commutes with
$\alpha$, $\lambda$ and $\omega$ and maps the start and end vertex of $X$
to the start and end vertex of $Y$ respectively. An \textit{isomorphism} is a
morphism which is bijective on both edges and vertices. A \textit{retraction} is
an idempotent morphism from a $\Sigma$-tree to itself; its image is called a \textit{retract}.
A tree is called \textit{pruned} if it does not admit a
non-identity retraction. (Structures without retractions are often called
\textit{cores} in graph theory.)

The $\Sigma$-tree with a single vertex and no edges is called \textit{trivial}.
The set of all isomorphism types of $\Sigma$-trees (including the trivial
$\Sigma$-tree) is denoted $UT^1(\Sigma)$
while the set of isomorphism types of non-trivial $\Sigma$-trees is
denoted $UT(\Sigma)$. The set of all isomorphism types of pruned trees
[respectively, non-trivial pruned trees] is denoted $T^1(\Sigma)$
[respectively, $T(\Sigma)$]. For any $X \in UT^1(\Sigma)$ there is a unique
$Y \in T^1(\Sigma)$ which is isomorphic to a retract of $X$ \cite[Proposition~3.5]{K_freeadequate};
we denote this pruned tree $\ol{X}$ and call it the \textit{pruning} of $X$.

If $X, Y \in UT^1(\Sigma)$ then the \textit{unpruned product} $X \times Y$ is (the isomorphism type of)
the tree obtained by glueing together $X$ and $Y$, identifying the end vertex
of $X$ with the start vertex of $Y$ and keeping all other vertices and
all edges distinct; this is a well-defined, associative binary operation
\cite[Proposition~4.2]{K_freeadequate}. If
$X \in UT^1(\Sigma)$ then $X^{(+)}$ is (the isomorphism type of) the tree with
the same labelled graph and start vertex of $X$, but with end vertex of $X^{(+)}$ the
start vertex of $X$.
Dually, $X^{(*)}$ is the isomorphism type of the idempotent tree with the
same underlying graph and end vertex as $X$, but with start vertex the end vertex of $X$.
We define corresponding \textit{pruned operations} on $T^1(\Sigma)$ by
$XY = \ol{X \times Y}$, $X^* = \ol{X^{(*)}}$ and $X^+ = \ol{X^{(+)}}$.

A tree with a single edge and distinct start and end vertices is called a
\textit{base tree}; we identify each base tree with the label of its
edge, thus viewing $\Sigma$ itself as a set of $\Sigma$-trees.
The main
result of \cite{K_freeadequate} is that $T^1(\Sigma)$ is the free adequate monoid
on $\Sigma$, being freely generated under pruned multiplication, $*$ and $+$ by the base $\Sigma$-trees
\cite[Theorem~5.16]{K_freeadequate}. The map
$X \to \ol{X}$ is a $(2,1,1,0)$-morphism from $UT^1(\Sigma)$ onto $T^1(\Sigma)$
\cite[Theorem~4.5]{K_freeadequate}.
Moreover, the submonoid of $T^1(\Sigma)$ generated by the base trees under pruned
multiplication and $*$ [respectively, $+$] is the free left adequate [respectively,
right adequate] monoid on $\Sigma$ \cite[Theorem~3.18]{K_onesidedadequate}. Free
adequate, left adequate or right adequate semigroups can all be obtained by discarding
the trivial tree (which is the identity element) in the corresponding monoids
(see \cite[Proposition~2.2]{K_freeadequate} and \cite[Proposition~2.6]{K_onesidedadequate}).

\section{Computing with Formulas and Trees}\label{sec_exptree}

In this section we study the computational complexity of converting between
well-formed formulas, over a generating set $\Sigma$ and the binary and unary
operations in an adequate semigroup, and $\Sigma$-trees. This will allow
us, in later sections, to use algorithms operating on $\Sigma$-trees to
solve computational problems involving formulas.

For our complexity
analysis throughout this paper, we shall work in the RAM
model of computation, in which integer operations and indirection
(finding a value stored at a known position in an array) take unit time.
For simplicity we will analyse 
the complexity of problems for a fixed rank semigroup, say on an alphabet
$\Sigma$, rather than the uniform complexity as the rank grows.
In places where it is necessary to be formal, we shall regard formulas as
words over the alphabet $\Omega$ consisting of generators from $\Sigma$ plus
the symbols
$($, $)$, $*$ and $+$ with the obvious meaning. We denote by $\Omega^*$
the set of all words over the alphabet $\Omega$, including the empty word
which we denote $\epsilon$. Our measure of the size of an expression will
be its length as a word over $\Omega$.

We assume
$\Sigma$-trees are by default stored as a natural number representing the
start vertex, a natural number representing the end vertex
and a list of edges (in no particular order), each
being a triple consisting of a label from $\Sigma$ and two natural numbers
encoding its start vertex and its end vertex. Sometimes it
will be expedient to convert trees to an alternative
representation. Note that the same abstract $\Sigma$-tree can admit
multiple representations, by numbering the vertices and ordering the edges
differently. Our measure of the size of a $\Sigma$-tree will be the
number of edges.

\begin{proposition}\label{prop_exptotree}
Given a well-formed formula $\omega$, one can compute in quadratic time
the (unpruned) $\Sigma$-tree which is its evaluation in $UT^1(\Sigma)$.
Moreover, this tree has size linear in the length of $\omega$.
\end{proposition}
\begin{proof}
A formula of length $n$ can be evaluated by a depth-first traversal
of a parse tree; this will clearly involve performing at most $n$ unpruned
operations with trees whose size is $O(n)$.
Clearly the unpruned $(+)$ and $(*)$ operations on trees can be performed
in constant time. Unpruned multiplication of trees can be performed in
time linear in the number of edges in the trees, by first relabelling
the vertices in the second tree (so that all references to its start
vertex become the end vertex of the first tree, and all its other vertices
are distinct from those in the first tree) and then concatenating the edge
lists and setting start and end vertices appropriately.

Thus, the $O(n)$ unpruned operations can each be performed in $O(n)$ time,
and the evaluation of the expression takes time $O(n^2)$.
Moreover, the resulting tree clearly has exactly one edge for each
occurrence of a generator in the expression, and hence has size linear
in the size of the expression.
\end{proof}

For our present purpose, the computations we wish to perform with
trees will all take quadratic time, so there is no particular benefit in
being able to compute the trees in faster than quadratic time.
However, we remark that the complexity of the algorithm given above can be
improved by a more sophisticated approach, using what is known in the computer science literature as a
``Union Find'' algorithm. Under this approach, when
performing multiplication, instead of merging the end vertex of one
tree with the start vertex of another, we keep them separate (allowing
the data structure to become a forest, rather than a tree) and
maintain another data structure recording which vertices are to be
merged at the end. An efficient implemention of this algorithm is extremely
close to being linear time; see \cite[Section~21.3]{Cormen01}
for more details.

Next, we shall show how a (not necessarily pruned) $\Sigma$-tree can be efficiently
converted into an well-formed formula.
We will define a function $\sigma : UT^1(\Sigma) \to \Omega^*$ such
that for each tree $X$, $\sigma(X)$ is a well-formed formula which
evaluates to $X$
in $UT^1(X)$, and then show that this function can be computed in quadratic time.
Note that since $\sigma$ is a function defined on abstract trees, the
algorithm produces a formula depending only on the abstract tree,
and not on its representation. We shall exploit this in
Section~\ref{sec_normalforms}
below
to compute normal forms (as formulas) for elements of the free adequate
semigroup.
To do this, we shall need a linear order on the set of all
formulas; for now, we will assume that we have such an order fixed. We
will discuss the choice and implementation of this order when we come to
analyse the complexity of the algorithm.

Let $X$ be a tree. We begin by defining a function $\rho$
from the vertex set of $X$ to $\Omega^*$; this is done inductively
by downwards induction on the distance of the vertex from the trunk.
Let $v$ be a vertex, and suppose $\rho$ is already defined on all vertices
strictly further from the trunk than $v$. Let
$v_1, \dots, v_p$ be the vertices adjacent to $v$ and strictly further from
the trunk, noting that $\rho(v_i)$ is already defined for each $i$.
For each $i$, let $e_i$ be the edge connecting $v$ to $v_i$,
and let $a_i \in \Sigma$ be its label. Define a formula $\tau_i \in \Omega^*$
by:
$$\tau_i = \begin{cases}
 (a_i \rho(v_i))+ &\textrm{ if $e_i$ is orientated away from $v$} \\
 (\rho(v_i) a_i)* &\textrm{ if $e_i$ is orientated towards $v$}
\end{cases}$$
Now we define $\rho(v) \in \Omega^*$ to be the word obtained by sorting the words
$\tau_i$ according to our ordering of formulas, and then concatenating.
(If $p = 0$, that is, if $v$ is a ``leaf'', this means $\rho(v) = \epsilon$.)

Now let $t_0, \dots t_q$ be
the trunk vertices of $X$ and $b_1, \dots, b_q$ the labels of the edges between
them, both in the obvious order. We define
$$\sigma(X) = \rho(t_0) b_1 \rho(t_1) b_2 \dots \rho(t_{q-1}) b_q \rho(t_q).$$

A simple but tedious inductive argument, akin to those in \cite{K_freeadequate},
shows that $\sigma(X)$ evaluates to the tree $X$ in $UT^1(\Sigma)$, and that
the number of characters in $\sigma(X)$ is at most four times the number
of edges in $X$.

To compute $\sigma(X)$, we start by precomputing adjacency matrices for $X$
corresponding to each possible edge label and direction; it is easily seen
that this can be done in $O(n^2)$ time where $n$ is the number of edges in $X$.
It is immediate from the inductive method of definition how to compute
$\sigma(X)$ by a simple depth first traversal (following non-trunk edges)
from each of the trunk vertices; this involves
considering each of $O(n)$ vertices once.

At each vertex, the only non-trivial operation is to sort the words $\tau_i$ into order
and then concatenate; the complexity of this of course depends on the choice of order. The
sum length of all the words $\tau_i$ is clearly $O(n)$. If
we choose the order to be lexicographic order (with respect to some arbitrary linear
order on $\Omega$), then a careful implementation of radix sort gives us a lexicographically
sorted list of formulas in $O(n)$ time, and concatenation is
clearly also $O(n)$.

Thus, the total time required for the algorithm is $O(n^2)$, and we have established:

\begin{proposition}\label{prop_treetoexp}
Given an unpruned $\Sigma$-tree $X$, we can in quadratic time compute
a well-formed formula which evaluates to $X$ in $UT^1(\Sigma)$. Moreover,
the formula has size linear in the size of $X$, and depends only
on the isomorphism type of $X$ and not on its representation.
\end{proposition}

\section{The Word Problem}\label{sec_wp}

Recall that the \textit{word problem} for an algebra $A$ with a given generating set
is the algorithmic problem of determining, given as input two well-formed formulas
over the generating set and the operations of the algebra, whether the formulas
represent the same element of the algebra. The word problem for free objects
in a variety of algebras is of particular importance, since it is trivially
equivalent
to the problem of testing whether a given identity holds in all algebras of
the variety.

In this section, we shall exhibit a quadratic time algorithm to solve
the word problem in a free adequate monoid $T^1(\Sigma)$. In fact in
Section~\ref{sec_normalforms} below, we shall see that it is also possible
to compute normal forms of elements of $T^1(\Sigma)$ in quadratic time;
this automatically yields another algorithm for the word
problem (by computing normal forms and comparing), of the same asymptotic
complexity. However, we present an explicit word problem algorithm first
since this is simpler, potentially easier to implement, and illustrates in
a simple context some of the ideas we will need in Section~\ref{sec_normalforms}.

By Proposition~\ref{prop_exptotree} we can efficiently convert
well-formed formulas in the free adequate monoid into unpruned
$\Sigma$-trees of comparable
size. It follows that to test
(efficiently) whether two given expression $x$ and $y$ represent the same
element of the free adequate monoid, that is, to solve the word problem,
it suffices to compute corresponding $\Sigma$-trees
$X, Y \in UT^1(\Sigma)$,
and then check (efficiently) if $\ol{X} = \ol{Y}$ in $T^1(\Sigma)$.

To solve
this latter problem, we begin with an elementary proposition, which
reduces it a constraint satisfaction problem (formulated in terms of
morphisms between structures, in the manner usual in the literature
of areas such as graph theory and universal algebra --- see for example
\cite{Hell04}).

\begin{proposition}\label{prop_equiv}
Let $X$ and $Y$ be $\Sigma$-trees. Then the following are equivalent:
\begin{itemize}
\item[(i)] $\ol{X} = \ol{Y}$;
\item[(ii)] $X$ and $Y$ admit isomorphic retracts;
\item[(iii)] there is a morphism from $X$ to $Y$ and a morphism from $Y$ to $X$.
\end{itemize}
\end{proposition}
\begin{proof}
The equivalence of (i) and (ii) follows from \cite[Proposition~3.5]{K_freeadequate}, so
it suffices to establish the equivalence of (ii) and (iii).

If (ii) holds then, in particular, some retract of $X$ is isomorphic to a
substructure of $Y$; composing the retraction of $X$ with the isomorphism
yields a morphism from $X$ to $Y$. By symmetry of assumption there is also
a morphism from $Y$ to $X$, so (iii) holds.

Now suppose (iii) holds, say $\sigma : X \to Y$ and $\tau : Y \to X$ are
morphisms. Then the compositions $\tau \circ \sigma : X \to X$ and
$\sigma \circ \tau : Y \to Y$ are maps on finite sets, and it follows that
we may choose $n$ such that both
$(\tau \circ \sigma)^n : X \to X$ and $(\sigma \circ \tau)^n : Y \to Y$
are idempotent, that is, are retractions of $X$ and $Y$ respectively.
Let $X'$ and $Y'$ be the retracts which are the respective images of
these retractions. Now it is easily verified that 
$\sigma$ and $\tau \circ (\sigma \circ \tau)^{n-1}$ restrict to mutually
inverse isomorphisms between the retracts $X'$ and $Y'$, showing that
(ii) holds.
\end{proof}

Proposition~\ref{prop_equiv} implies that to check if two $\Sigma$-trees
are equivalent, and hence by the preceding arguments to solve the word problem
for the free adequate semigroup on $\Sigma$, it suffices to check whether each
$\Sigma$-tree admits a morphism to the other. Our main goal in the rest of
this section, then, is an efficient algorithm to test, given an ordered pair of
$\Sigma$-trees, whether there is a morphism from the first to the second.
Our approach is essentially a constraint propagation
algorithm, with the correctness of the result being shown by an 
arc consistency argument utilising the tree-like nature of
our geometric representatives for elements. The ideas behind the proof
are well known in the fields of constraint satisfaction and artificial
intelligence (see for example
\cite{Dechter87}), but for the benefit of semigroup theorists who may not be
familiar with these fields we present the algorithm in an elementary form:

\begin{algorithm}\label{algH} \ \\

\noindent\textbf{Input:} Two $\Sigma$-trees $T_1$ and $T_2$ on $n$ and $m$ vertices respectively.

\noindent\textbf{Output:} ``Yes'' if there exists a homomorphism from $T_1$ to $T_2$. ``No'' otherwise.

\begin{enumerate}
\item Consider the start vertex of $T_1$, label this vertex $1$, and then use
a depth-first traversal (ignoring direction of edges) to label the remaining vertices
from $2$ to $n$ in ascending order.

\item For each $i$ in $\{1,\dots,n\}$, let $B_i$ be the set of vertices
in $T_2$.

\item For the start [end] vertex $i$ set $B_i$ to be the singleton set containing the
start [end] vertex of $T_2$
\item For $i$ descending from $n$ to $1$, and each vertex $j > i$ adjacent to $i$, do the following:
\begin{itemize}
\item[(i)] Let $a \in \Sigma \cup \Sigma'$ be the label of the edge from $i$ to $j$ in $T_1$;
\item[(ii)] Let $B_i:= B_i \cap B_j^\star$ where
\begin{align*}
B_j^\star = \lbrace x \mid & \textrm{ $T_2$ has an edge labelled $a$ } \\
&\textrm{ from $x$ to some $y \in B_j$ } \rbrace.
\end{align*}
\end{itemize}
\item If $B_1=\emptyset$, output ``No''; otherwise output ``Yes.''
\end{enumerate}
\end{algorithm}

\begin{proposition}\label{prop_wpcorrect}
Algorithm~\ref{algH} is correct, that is, $B_1$ is non-empty on completion of the
algorithm if and only if there is a morphism from $T_1$ to $T_2$.
\end{proposition}
\begin{proof}
For brevity, we identify the vertices with the labels from $1$ to $n$
assigned in the algorithm.
Suppose first that there is a morphism $\sigma : T_1 \to T_2$. We claim
that $B_i$ contains $\sigma(i)$ for all $i$, from which it follows in
particular that $B_1$ contains $\sigma(1)$ so that $B_1$ is
non-empty as required. Indeed, if not, choose $i$ maximal such that
$\sigma(i) \notin B_i$. Clearly $\sigma(i)$ was in $B_i$ after Step 2 of the
algorithm and, because $\sigma$ preserves start and end vertices, also after
Step 3; therefore, it must have been removed during Step 4.
For this to have happened, there must have been a $j > i$ and an edge from $i$
to $j$ (labelled $a \in \Sigma \cup \Sigma'$, say) such that
$\sigma(i) \notin B_j^\star$.
By the definition of $B_j^\star$, this means there was (at the time of removal)
no edge labelled $a$ from $\sigma(i)$ to any $y \in B_j$.
 But because $\sigma$ is a morphism, $\sigma(j) \in B_j$
is connected to $\sigma(i)$ by such an edge, so it must be that
$\sigma(j)$ was not in $B_j$ at the time $\sigma(i)$ was removed from
$B_i$. Now since $B_j$ only gets smaller, $\sigma(j)$ is not in $B_j$
at the end of the algorithm. But $j > i$, so this contradicts the
maximality of $i$.

Conversely, suppose $B_1$ is non-empty at the end of the algorithm. We
define a morphism $\sigma : T_1 \to T_2$ inductively as follows. First,
choose $\sigma(1) \in B_1$ arbitrarily. Now assume $1 < i < n$ and
we have defined
$\sigma$ on the vertices $1$, \dots $i-1$ and all edges between them, in
such a way as to preserve adjacency, labels and directions of edges, and
the start and end vertices if appropriate, and such that $\sigma(p) \in B_p$
for $1 \leq p \leq i-1$.

Since $T_1$ is a tree and the edges were numbered by a depth-first 
traversal, it follows that vertex $i$ is connected to vertex $k$ for some 
unique $k < i$; suppose $T_1$ has an edge from $k$ to $i$ labelled
$a \in \Sigma \cup \Sigma'$. 

Considering the way $B_k$ is constructed, we see that every vertex in $B_k$,
including $\sigma(k)$, is connected to some $v \in B_i$. 
Moreover, if $i$ happens to be the start [respectively, end] vertex of $T_1$,
then $B_i$ was originally set to contain only the start [end] vertex of $T_2$,
so it must be that $v \in B_i$ is the start [end] vertex of $T_2$.
Thus, by defining $\sigma(i) = v$ and $\sigma(e)$ to be the 
appropriate edge, we extend $\sigma$ to be defined on the vertices 
$1, \dots, i$ and all edges between, with the appropriate properties.
\end{proof}

We now analyse the complexity of Algorithm~\ref{algH}.
At the start of the algorithm, we can precompute for
each vertex in $T_1$ a list of edges adjacent to that vertex; this can be done
in $O(n)$ time.

Having done this, Step 1 of the algorithm (a simple depth first traversal
of the tree $T_1$) has complexity $O(n)$. If we store the lists $B_i$ as
arrays of $m$ boolean flags then Step 2 has complexity $O(mn)$
since we need to initialise $mn$ values. Step 3 has complexity $O(m)$,
since we must reset $m-1$ values for each of the start and end vertices.

The most interesting part is the complexity of Step 4. The number of
iterations of the outer loop is clearly bounded by the number of edges
in $T_1$, so it is $O(n)$ and the precomputed lists of edges mean there is
no extra overhead in finding the edges in the correct order.
In each iteration, the fact that the corresponding edge has been found
means Step 4(i) takes constant time. In Step 4(ii), computing $B_j^\star$
involves
passing through the list of all $O(m)$ edges of $T_2$ and for each edge
checking (in constant time) if one of the ends lies in $B_j$ and if the
label is correct; this takes
$O(m)$ time. Computing the intersection is simply
a boolean ``and'' operation on two arrays of length $m$, and so also takes $O(m)$
time. Thus, Step 4 takes time $O(mn)$, and the total complexity of the
algorithm is $O(mn)$.

Combining the above arguments with Proposition~\ref{prop_exptotree},
we have established the following main result:
\begin{theorem}
The word problem for any finite rank free left adequate, free right
adequate or free adequate semigroup is decidable in time
polynomial (quadratic, in the RAM model of computation) in the combined
length of the two formulas.
\end{theorem}

\section{Pruned Trees and Normal Forms}\label{sec_normalforms}

In this section, we show how to efficiently compute the minimal retract of a 
$\Sigma$-tree. Combined with the results of Section~\ref{sec_exptree}, this
will allow us to compute normal forms (as formulas) for elements of
free adequate monoids. Our main algorithm is the following, the first four
steps of which are essentially the same as in Algorithm~\ref{algH}:

\begin{algorithm}\label{algN} \ \\

\noindent\textbf{Input:} A $\Sigma$-tree $T$ on $n$ vertices.

\noindent\textbf{Output:} The vertex set of a pruned subtree of $T$,
isomorphic to the $\ol{T}$.

\begin{itemize}
\item[(1)] Consider the start vertex of $T$, label this vertex $1$, and then use
a depth-first traversal (ignoring direction of edges) to label the remaining vertices
from $2$ to $n$ in ascending order.

\item[(2)] For each $i$ in $\{1,\dots,n\}$, set $B_i = \lbrace 1, \dots, n \rbrace$.

\item[(3)] For the start [end] vertex $i$ set $B_i = \lbrace i \rbrace$.

\item[(4)] For $i$ descending from $n$ to $1$ and each $j$ with $j > i$ and $i$ connected to $j$, do the following:
\begin{itemize}
\item[(i)] Let $a \in \Sigma \cup \Sigma'$ be the label of the edge in $T$
from $i$ to $j$.
\item[(ii)] Let $B_i := B_i \cap B_j^\star$ where
\begin{align*}
B_j^\star = \lbrace x \mid & \textrm{ $T$ has an edge labelled $a$ } \\
&\textrm{ from $x$ to some $y \in B_j$ } \rbrace.
\end{align*}
\end{itemize}

\item[(5)] Set $X = \lbrace 1, \dots, n \rbrace$.

\item[(6)] For $w$ ascending from $1$ to $n$ and $a \in \Sigma \cup \Sigma'$,
do the following:
\begin{itemize}
\item[(i)] If $w \notin X$ then go to the next $w$.
\item[(ii)] Otherwise, find all vertices $u$ such that $a$ labels
an edge from $w$ to $u$ and put them in a list $K$.
\item[(iii)] For each $u \in K$ such that $u > w$:
\begin{itemize}
\item[(a)] Check if $K \cap B_u = \lbrace u \rbrace$.
\item[(b)] If \textbf{not}, then remove $u$ from $K$, and traverse the
tree below $u$, removing $u$ and all its descendant vertices from $X$.
\end{itemize}
\end{itemize}
\item[(7)] Output $X$.
\end{itemize}
\end{algorithm}

Our next aim is to prove the correctness of this algorithm.

\begin{lemma}\label{lemma_retract}
The subtree $X$, as computed at the end of Algorithm~\ref{algN}, is a retract of $T$.
\end{lemma}
\begin{proof}
We shall show that each time a vertex and its descendants are removed from
$X$ at Step 6(iii)(b), there is a retraction from the tree $X$ prior
to the removal, onto the tree $X$ after the removal. Since the successive
subtrees $X$ form a chain under inclusion, it is clear that composing these
retractions in the appropriate order yields a retraction from $T$ onto
the final tree $X$, as required.

Indeed, suppose $u$ and its descendants are removed from $X$ at some point.
Let $w$, $a$ and $K$ be as in the algorithm at that point, and let $X_1$ and $X_2$
be the values of $X$ immediately before and after
the deletion, respectively.

Note that, since the identity map is a morphism, it is easily verified
that $i \in B_i$ for all vertices $i$ of $T$. The
fact that $u$ was removed means that $K \cap B_u \neq \lbrace u \rbrace$, and
we know $u \in K \cap B_u$, so we may choose some vertex $v \in K \cap B_u$
with $v \neq u$.

First, we follow the procedure from the proof of Proposition~\ref{prop_wpcorrect}
to inductively define a morphism $\sigma : T \to T$, but
being  more careful about our choices in order to ensure that
$\sigma(u) = \sigma(v) = v$.
We start by setting $\sigma(i) = i$ for all $i < u$; since $i \in B_i$ for all
$i$ it is easily verified that this is consistent with the procedure
in Proposition~\ref{prop_wpcorrect}. Note in particular
that $w < u$, so this means $\sigma(w) = w$. Now since $u, v \in K$, there
are edges from $w = \sigma(w)$ to $u$ and $v$ both labelled $a$, so in
following the procedure of Proposition~\ref{prop_wpcorrect}
we may choose to set $\sigma(u) = v$.
We now continue the process from the proof of Proposition~\ref{prop_wpcorrect}. For each vertex
$r$ in turn, if $r$ is a descendant of $u$, then we define $\sigma(r)$
as in the proof of Proposition~\ref{prop_wpcorrect}, making any choices arbitrarily. If $r$ is
not a descendant of $u$ then the unique vertex $k < r$ adjacent to $r$
is also not a descendant of $u$; thus, we have already defined
$\sigma(k) = k$ and we may set $\sigma(r) = r$.

Now $\sigma$ is a map on a finite set, and so has an idempotent power,
say $\sigma^i$. Since $v$ is not a descendant of $u$, we have
$\sigma(v) = v$, and hence $\sigma^i(u) = \sigma^{i-1}(v) = v$, so $u$
is not in the image of $\sigma^i$. Since the image of $\sigma^i$ is a
$\Sigma$-tree, it must contain the start
vertex and be connected, so we deduce that no descendants of $u$ are
in the image of $\sigma^i$. It follows that $\sigma^i$ maps $X_1$ to
$X_2$. Moreover, $\sigma$ fixes $X_2$, so restricting $\sigma^i$ to
$X_1$ gives the required retraction of $X_1$ onto $X_2$.
\end{proof}

\begin{lemma}
The retract $X$, as computed at the end of Algorithm~\ref{algN}, is pruned.
\end{lemma}
\begin{proof}
Suppose not, say $X$ admits a proper retraction $\sigma : X \to X$. Let
$u$ be a vertex in $X$ but not in the image of $\sigma$, and suppose $u$
is minimal with respect to this condition. Then $u \neq 1$, since $1$ labels
the start vertex which is fixed by every retraction. Thus, we may let $w$ be
the unique vertex with $w < u$ and $w$ adjacent to $u$.

Since $w < u$, by the minimality
of the choice of $u$, we have $\sigma(w) = w$. It follows that $\sigma(u)$
is connected to $w$ by an edge of the same label and orientation as that
connecting $u$ to $w$. This means that, when considering $w$ at step
6(iii), we would initially have had $\sigma(u) \in K$. Since
$\sigma(u)$ is in the final tree $X$, it was never removed from $K$.
Moreover, composing the retraction of $T$ onto $X$ (given by
Lemma~\ref{lemma_retract}) with
$\sigma$ gives a morphism of $T$ mapping $u$ to $\sigma(u)$; it follows
from the argument in the proof of Proposition~\ref{prop_wpcorrect} that
$\sigma(u) \in B_u$.

This means that at the time
$u$ was considered in Step 6(iii)(a) we had
$\sigma(u) \in K \cap B_u$. But then $K \cap B_u \neq \lbrace u \rbrace$,
so $u$ would have been removed from $X$, giving a contradiction.
\end{proof}

Turning to the complexity of the algorithm, Steps (1)-(4) are exactly
as in Algorithm~\ref{algH} (except that the source and target trees for
the morphism are the same, so $m=n$), and by the same analysis as in
Section~\ref{sec_wp} take time $O(n^2)$.

For efficiency, we store the set $X$ as an array of boolean flags. The time
requirement for Step
(5) is clearly $O(n)$. The loop in Step 6 is iterated at most $O(n)$ times.
In each such iteration, step (i) takes constant time. Step (ii) cannot
involve checking more than $O(n)$ vertices, so the total contribution
to the time required will be $O(n^2)$. In step (iii), note that each
element of $L$ is uniquely determined (across the entire algorithm) by
the ordered pair $(w,u)$ where there is always an edge between $w$ and $u$;
thus, the number of iterations of this step across the whole algorithm is
at most twice the number of edges in the tree, which is $O(n)$.
Within each iteration, each step takes $O(n)$ time, so the total contribution is
$O(n^2)$.

Thus, we have established:
\begin{theorem}\label{thm_pruning}
Given a $\Sigma$-tree $T$, one can compute in polynomial time (quadratic
time in the RAM model of computation) the pruned $\Sigma$-tree $\ol{T}$.
\end{theorem}

Combining with the results of Section~\ref{sec_exptree}, Theorem~\ref{thm_pruning}
allows us to compute normal forms (as formulas) in the free adequate monoid.
Indeed, given a formula $w$, by Proposition~\ref{prop_exptotree} we may
convert it in quadratic time to a corresponding unpruned $\Sigma$-tree $T$ of comparable size. By
Theorem~\ref{thm_pruning} we may then compute the pruned tree $\ol{T}$ in
time quadratic in the size of $T$ and hence in the size of $w$. Finally, by
Proposition~\ref{prop_treetoexp} we can convert $\ol{T}$ into the uniquely
defined formula $\sigma(\ol{T})$ in time quadratic in the size of $\ol{T}$; since $\ol{T}$
is no larger than $T$, this is also quadratic in the size of $T$, and hence
in the size of $w$.

\begin{theorem}
Given a formula in the free adequate, left adequate or right adequate
semigroup or monoid, one can compute
a normal form in polynomial (quadratic in the RAM model of computation) time.
\end{theorem}

We note that the resulting language of normal forms for elements, which by
definition is the set
$$\lbrace \sigma(\ol{T}) \mid T \textrm{ is a pruned $\Sigma$-tree} \rbrace,$$
does not appear to have a completely elementary description without reference to trees. Of
course one may check (in quadratic time) whether a given formula $w$ is a
normal form by following the above procedure to convert $w$ to a normal form
and then comparing with $w$; we do not know of a fundamentally easier method.

We also note that in the case of free \textit{inverse} monoids (and semigroups), it is known
\cite[Theorem 11]{Lohrey07} that the word problem is decidable in (RAM) linear 
time. In the inverse case computations appear to be inherently simpler, as
the operation corresponding to computing a minimal retract (namely, computing a minimal
morphic image) can be performed by an iterative process of identifying vertices, where
the fact a pair of vertices can be identified is determined ``locally'', by looking only
in the immediate neighbourhood of the vertices. It seems unlikely that quite such a
fast algorithm can be obtained in the adequate case, but one might still ask whether our
algorithms can be significantly improved upon. Also shown in \cite[Theorem 11]{Lohrey07}
is that the word problem for a free inverse monoid is decidable (using a different
algorithm to the linear time one) in logarithmic space: the space complexity of the word
problem for free adequate monoids and semigroups is a natural topic for future research.

\section*{Acknowledgements}

The second author was supported by the Czech Government Grant Agency GA\v CR
project 13-01832S.
The authors thank the organisers of the 4th Novi Sad Algebraic Conference
(NSAC2013), which by bringing together researchers in semigroup theory
and universal algebra catalysed this research. They also thank Victoria Gould
for some helpful comments on the draft, and Stuart Margolis for pointing them
to the work of Lohrey and Ondrusch \cite{Lohrey07}.

\bibliographystyle{plain}

\def\cprime{$'$} \def\cprime{$'$}

\bibliography{mark}
\end{document}